\documentclass[12pt]{article}
\usepackage{amsmath,amsfonts,amsthm,amscd}
\usepackage{latexsym}

\newtheorem{thm}{Theorem}[section]
\newtheorem{pro}[thm]{Proposition}
\newtheorem{cor}[thm]{Corollary}
\newtheorem{lem}[thm]{Lemma}

\newtheorem{rem}[thm]{Remark}

\newcommand{\Hom}{{\rm Hom}}
\newcommand{\Aut}{{\rm Aut}}
\newcommand{\End}{{\rm End}}

\newcommand{\di}{{\rm d}}
\newcommand{\Soc}{{\rm Soc}}
\newcommand{\Jac}{{\rm Jac}}
\newcommand{\M}{{\rm M}}
\newcommand{\C}{{\rm C}}
\newcommand{\K}{{\rm K}}

\newcommand{\cA}{{\mathcal{A}}}
\newcommand{\cB}{{\mathcal{B}}}
\newcommand{\cC}{{\mathcal{C}}}
\newcommand{\cE}{{\mathcal{E}}}
\newcommand{\cD}{{\mathcal{D}}}
\newcommand{\cF}{{\mathcal{F}}}
\newcommand{\cG}{{\mathcal{G}}}
\newcommand{\cL}{{\mathcal{L}}}
\newcommand{\cM}{{\mathcal{M}}}
\newcommand{\cN}{{\mathcal{N}}}
\newcommand{\cS}{{\mathcal{S}}}
\newcommand{\cP}{{\mathcal{P}}}
\newcommand{\cQ}{{\mathcal{Q}}}
\newcommand{\cT}{{\mathcal{T}}}
\newcommand{\cR}{{\mathcal{R}}}
\newcommand{\cU}{{\mathcal{U}}}

\newcommand{\fm}{{\mathfrak{m}}}

\newcommand{\ds}{\displaystyle}

\begin{document}

\title{Intersection Graph of a Module}

\author{Erg\"un Yaraneri}
\maketitle

\small
\begin{center}
Department of Mathematics, Istanbul Technical University,
34469 Maslak, Istanbul, Turkey. \\
{\it e-mail:} eyaraneri@yahoo.com
\end{center}

\begin{abstract}
Let $V$ be a left $R$-module where $R$ is a (not necessarily commutative) ring with unit. The intersection graph $\cG(V)$ of proper $R$-submodules of $V$ is an undirected graph without loops and multiple edges defined as follows: the vertex set is the set of all proper $R$-submodules of $V,$ and there is an edge between two distinct vertices $U$ and $W$ if and only if $U\cap W\neq 0.$ We study these graphs to relate the combinatorial properties of $\cG(V)$ to the algebraic properties of the $R$-module $V.$ We study connectedness, domination, finiteness, coloring, and planarity for $\cG (V).$ For instance, we find the domination number of $\cG (V).$ We also find the chromatic number of $\cG(V)$ in some cases. Furthermore, we study cycles in $\cG(V),$ and complete subgraphs in $\cG (V)$ determining the structure of $V$ for which $\cG(V)$ is planar.


\end{abstract}

\section{Introduction}
Let $\cF $ be a set consisting of proper subobjects of an object with an algebraic structure. Examples of $\cF $ include the set of proper subgroups of a finite group, the set of proper subspaces of a vector space, and the set of proper ideals of a commutative ring. One may define an undirected graph, containing no loops and no multiple edges, constructed from $\cF $ as follows: there is a vertex for each subobject in $\cF ,$ and there is an edge between two vertices whenever the intersection of the subobjects representing the vertices is not the zero object, where the zero object denotes the object having a unique endomorphism. For instance, if $\cF $ is the set of proper subgroups of a finite group then the zero object is the trivial subgroup consisting of only the identity of the group. The graph constructed above is usually called the intersection graph of $\cF ,$ and it is studied in many papers. The intersection graphs of the proper subgroups of a finite group, the proper subspaces of a finite dimensional vector space over a finite field, certain affine subspaces, and the proper ideals of a commutative ring are studied in, for instance, \cite{Chak,Jaf,Jaf2,Jaf3,Lai,Shen,Zel} and some of the references there.

In this paper, we study the intersection graph of the proper submodules of any module over any ring. Therefore, most of the results in the some of the above papers will be easy consequences of some of the results obtained here. The connectivity and the clique number of the intersection graph of the proper submodules of a module are studied in \cite{Ak}. Here we study more aspects of the intersection graph of the proper submodules of a module and obtain more results.

Throughout the paper $R$ is a unital ring which is not necessarily commutative and $V$ is a unitary left $R$-module. By a proper $R$-submodule of $V$ we mean an $R$-submodule of $V$ different from $0$ and $V.$ We denote by $\cG (V)$ the intersection graph of the proper $R$-submodules of $V.$ Therefore, $\cG (V)$ is an undirected graph without loops and multiple edges defined as follows: the vertex set is the set of all proper $R$-submodules of $V,$ and there is an edge between two distinct vertices $U$ and $W$ if and only if $U\cap W\neq 0.$

In Sections 2 and 3, we determine the algebraic structure of the module $V$ for which the graph $\cG (V)$ has any one of the following properties: it is connected; it has a cut vertex; it has a cut edge; it has an $n$-cycle; it is bipartite. Most of the results in Sections 2 and 3 are easy to derive, we include them for completeness.

As the number of $R$-submodules of $V$ is usually not finite, the graph $\cG (V)$ will usually have infinitely many vertices. Most of the notions and results coming from finite graph theory are not true for infinite graphs, see, for instance, \cite{Bas,Ore}. However, we observe that this is not a problem for the study of the graph $\cG (V).$ For example, in Section 4, we justify that $\cG (V)$ has a minimal dominating set for any $V,$ and we determine the domination number of $\cG (V)$ for any $V,$ which may be an infinite cardinal number.

In Section 5, we find the number of the maximal $R$-submodules of $V,$ from which we obtain the structure of $V$ for which $\cG (V)$ has finitely many vertices. We also obtain more results about the number of submodules in Section 6.

We study the chromatic number of $\cG (V)$ in Section 7. For instance, we determine the structure of $V$ for which the chromatic number of $\cG (V)$ is finite. We also calculate the chromatic number of $\cG (V)$ when $V$ is a semisimple module having a (finite) composition series. This result is still general enough to cover the known results (about the chromatic number of the intersection graph of the proper subspaces of a vector) as special cases.

In the last section, we determine the structure of $V$ for which the graph $\cG (V)$ is planar.

To explain our notations and terminologies, let $W$ be an $R$-module. We denote by $\Jac (W)$ and $\Soc (W)$ the Jacobson radical of $W$ and the socle of $W,$ respectively. By a section of $W$ we mean a quotient module of a submodule of $W,$ so sections of $W$ are of the form $Y/X$ where $X\subseteq Y$ are $R$-submodules of $W.$ For any natural number $n,$ the direct sum of $n$ modules all of which are $W$ is denoted by $nW.$ The sets of $R$-module endomorphisms and automorphisms of $W$ are denoted by $\End _R(W)$ and $\Aut _R(W),$ respectively. We write $U\leq W$ to indicate that $U$ is an $R$-submodule of $W.$ Moreover, for $R$-submodules $U'$ and $U$ of $W,$ we say that $U'$ is a complement of $U$ in $W$ if $U\cap U'=0$ and $U+U'=W.$

Let $\cG $ be a graph. If there is a finite path in $\cG $ from a vertex $v$ to a vertex $w,$ then we use the ordered tuple $(u_1,u_2,...,u_n)$ to denote this path, where $u_1=v$ and $u_n=w,$ and each $u_i$ is a vertex on the path, and for any $i$ with $1\leq i\leq n-1$ there is an edge on the path between $u_i$ and $u_{i+1}.$

\section{Connectedness}
For any distinct proper $R$-submodules $X$ and $Y$ of $V,$ we let $\di _V(X,Y)$ denote the distance between $X$ and $Y,$ which is defined to be the length of the shortest path in $\cG (V)$ between $X$ and $Y.$ By the length of a path we mean the cardinality of the set of all the edges on the path. Although the definition of the distance above may be meaningless for a graph with infinitely many vertices and edges, we see in this section that if there is a path in $\cG (V)$ between two vertices then we may choose one with the length less than or equal to $2.$ In a graph it may happen that there is no path between some vertices, in which case, the distance is undefined, and such graphs are called disconnected. Moreover, in an infinite graph the distance between two vertices may be an infinite cardinal number. However, in the intersection graph of a module we observe in this section that the distance between any two distinct vertices is a finite number. We define the diameter of a connected graph $\cG $ to be the largest of the distances between distinct vertices of $\cG ,$ if such a largest distance exists. If $\cG $ has no distinct vertices, we assume that the diameter is $0.$ See, for instance, \cite{Ore} and \cite{West}.

\begin{lem}
\label{2.1} Let $X$ and $Y$ be distinct proper $R$-submodules of $V.$ Then, there is no path in $\cG (V)$ between $X$ and $Y$ if and only if $X$ and $Y$ are simple $R$-submodules of $V,$ and $V=X\oplus Y,$ direct sum of $X$ and $Y.$
\end{lem}
\begin{proof} Suppose that there is no path in $\cG (V)$ between $X$ and $Y.$ In particular $X\cap Y=0.$ Take any nonzero submodule $Z$ of $X.$ Then $Z\cap Y=0.$ If $Z+Y\neq V,$ then $(X,Z+Y,Y)$ is a path of length $2.$ Hence, $Z\cap Y=0$ and $Z+Y=V$ for any nonzero submodule $Z$ of $X.$ In particular, $V$ is the direct sum $X\oplus Y$ of $X$ and $Y.$ Moreover, the modular law shows that $X$ is simple. See, for instance, \cite[page 71]{Warf} for the modular law.

Conversely, if $V$ is a direct sum of two simple modules, then any proper submodule of $V$ is simple so that $\cG (V)$ has no edges in this case.
\end{proof}

The proof of the previous result implies the following.

\begin{rem}
\label{2.2} Let $X$ and $Y$ be distinct proper $R$-submodules of $V.$ If there is a path in $\cG (V)$ between $X$ and $Y,$ then $\di _V(X,Y)\leq 2.$
\end{rem}

The following is an obvious consequence of \ref{2.1}, and it generalizes the related results of \cite{Chak}.
\begin{pro}
\label{2.3} Suppose that $V$ is not a simple $R$-module. Then, $\cG (V)$ is connected if and only if either $V$ is not a semisimple $R$-module, or $V$ is a semisimple $R$-module containing a direct sum of three simple $R$-modules.
\end{pro}
The previous result is a restatement of \cite[Theorem 2.1]{Ak}, which is obtained in a slightly different way.

\begin{cor}
\label{2.4} If $\cG (V)$ has an edge then $\cG (V)$ is connected.
\end{cor}
\begin{proof} Suppose that $\cG (V)$ has an edge. Then, $V$ must have two distinct proper submodules such that at least one of them must be not simple. The result follows from \ref{2.3}.
\end{proof}

\begin{rem}
\label{2.5} $\cG (V)$ has no edges if and only if $V$ has a composition series of length less than or equal to $2.$
\end{rem}
\begin{proof} Suppose that $\cG (V)$ has no edges. If $V$ has an ascending or descending chain of $n$ proper submodules, then there will be an edge between any two of these $n$ submodules. This shows that $V$ has a composition series of length less than or equal to $2.$ The converse is clear, because in this case if there is a proper submodule of $V$ then it must be simple.
\end{proof}

As an easy consequence of \ref{2.2} we state that if $\cG (V)$ is connected, then the possibilities for the diameter of $\cG (V)$ are $0,$ $1,$ and $2.$ The diameter of $\cG (V)$ is $0$ or $1$ if and only if the intersection of any two nonzero $R$-submodules of $V$ is nonzero (equivalently, every nonzero $R$-submodule of $V$ is indecomposable). Such a module is called a uniform module. See, for instance, \cite{Lam,Warf} for uniform modules.

Let $v$ be a vertex of a graph $\cG .$ By $\cG -v$ we mean the graph obtained from $\cG $ by removing the vertex $v$ and all the edges adjacent to $v.$ If $\cG -v$ has more connected components than $\cG $ has, then $v$ is called a cut vertex of $\cG .$

\begin{rem}
\label{2.6} $\cG (V)$ has a cut vertex if and only if the socle of $V$ is a direct sum of two simple $R$-modules and is a maximal $R$-submodule of $V.$ Moreover, in this case, the socle of $V$ is the unique cut vertex of $\cG (V).$
\end{rem}
\begin{proof} Let $X$ be a cut vertex of $\cG (V).$ There must be two distinct vertices $A$ and $B$ of $\cG (V)-X$ such that there is a path in $\cG (V)$ between $A$ and $B,$ but there is no path in $\cG (V)-X$ between $A$ and $B.$ Hence, $A\cap B=0,$ and \ref{2.2} implies that $(A,X,B)$ is a path in $\cG (V).$ In particular, $A\cap X\neq 0$ and $B\cap X\neq 0.$ Moreover, $X$ must be a maximal $R$-submodule of $V,$ because, for any proper submodule $Y$ of $V$ containing $X$ properly, $(A,Y,B)$ is a path in $\cG (V)-X$ between $A$ and $B.$ To see that $X$ is a direct sum of two simple modules, we take any nonzero submodule $A'$ of $A\cap X$ and any nonzero submodule $B'$ of $B\cap X.$ If $A'+B'\neq X,$ then $(A,A'+B',B)$ is a path in $\cG (V)-X$ between $A$ and $B.$ Hence, $X=A'+B',$ implying that $X$ is the direct sum of the modules $A\cap X$ and $B\cap X,$ which are necessarily simple as a consequence of the modular law. For the modular law see, for instance, \cite[page 71]{Warf}. Finally, we want to see that $X$ is equal to the socle of $V.$ Otherwise, there is a simple submodule $S$ of $V$ not contained in $X.$ But then, $$(A,(A\cap X)+S,(B\cap X)+S,B)$$ is a path in $\cG (V)-X.$

Conversely, letting $Z$ be the socle of $V,$ we suppose that $Z=U\oplus T$ is the direct sum of two simple modules $U$ and $T$, and $Z$ is a maximal submodule of $V.$ The vertices $U$ and $T$ are connected by the path $(U,Z,T)$ in $\cG (V).$ We will observe that there is no path in $\cG (V)-Z$ between $U$ and $T,$ implying that $Z$ is a cut vertex. Suppose for a moment that there is a path in $\cG (V)-Z$ between $U$ and $T.$ Let $U_1$ and $T_1$ be the vertices on this path such that $U_1$ is adjacent to $U$ and $T_1$ is adjacent to $T$ by some edges on the path. Then, $U\cap U_1\neq 0$ and $T\cap T_1\neq 0,$ and as $U$ and $T$ are both simple, $U\leq U_1$ and $T\leq T_1.$ Since $V$ has a composition series of length $3,$ we see that $U_1$ and $T_1$ are maximal submodules of $V,$ and the intersection of any two distinct maximal submodules of $V$ is simple. Therefore, $Z\cap U_1=U$ and $Z\cap T_1=T,$ implying that $0=U\cap T=(U_1\cap T_1)\cap Z.$ But, being simple, $U_1\cap T_1$ must be a submodule of the socle $Z,$ from which the last equality implies that $U_1\cap T_1=0.$ This is impossible because $U_1\cap T_1$ is a simple module.
\end{proof}

Let $e$ be an edge of a graph $\cG .$ By $\cG -e$ we mean the graph obtained from $\cG $ by removing the edge $e.$ But we keep the endpoints of $e$ so that the vertices of  $\cG -e$ and $\cG $ are the same. If $\cG -e$ has more connected components than $\cG $ has, then $e$ is called a cut edge of $\cG .$

\begin{rem}
\label{2.7} $\cG (V)$ has a cut edge if and only if $V$ has a composition series of length $3,$ and $V$ has a simple $R$-submodule that is contained in a unique maximal $R$-submodule of $V.$ Moreover, in this case, the cut edges in $\cG (V)$ are precisely the edges incident to the mentioned simple modules.
\end{rem}
\begin{proof} Let $e$ be a cut edge of $\cG (V)$ connecting the vertices $A$ and $B.$ There must be two vertices such that every path between them contains the edge $e.$ Thus, there must be no path between $A$ and $B$ other than $e.$ As $A\cap B\neq 0,$ if $A,B,$ and $A\cap B$ are all distinct then $(A,A\cap B,B)$ is a path between $A$ and $B.$ Hence, we may assume that $B< A.$ If there is a submodule $X$ of $V$ properly lying between the submodules $0$ and $B,$ then $(A,X,B)$ is a path between $A$ and $B.$ This shows that $B$ must be simple. Similarly, we see that $A/B$ and $V/A$ are also simple. So far, we have shown that $0<B<A<V$ is a composition series of $V$ of length $3.$

If $J$ is a maximal submodule of $V$ other than $A$ that contains $B,$ then $(A,J,B)$ is a path between $A$ and $B.$ Hence, the only maximal submodule of $V$ containing $B$ is $A.$

Conversely, suppose $V$ has a composition series of length $3$ and it has a simple submodule $S$ that is contained in a unique maximal submodule $W$ of $V.$ We will show that the edge $f$ between $S$ and $W$ is a cut edge. Suppose for a moment that $f$ is not a cut edge. There must be a path other than $f$ between $S$ and $W.$ Let $T$ be the vertex on this path which is adjacent to $S$ by an edge on the path. In particular, $T\neq W.$ As $S$ is simple, we see that $S<T.$ Since the composition length of $V$ is $3,$ it now follows that $T$ is a maximal submodule of $V$ containing $S.$
\end{proof}

\section{Cycles}
In this section we study the existence of cycles in $\cG (V).$ We begin with an easy observation.

\begin{lem}
\label{3.1} Let $m\geq 3$ be a natural number. If $\cG (V)$ has no $m$-cycles, then $V$ has a composition series of length less than or equal to $m.$
\end{lem}
\begin{proof} If $V$ has an ascending or descending chain of $n$ proper submodules, then it is clear that the subgraph of $\cG (V)$ induced by these $n$ submodules is the complete graph. So, in particular, there is an $n$-cycle in $\cG (V).$ The result follows.
\end{proof}

\begin{pro}
\label{3.2} The following conditions are equivalent:
\begin{description}
  \item[{\rm (i)}] $\cG (V)$ has no cycles.
  \item[{\rm (ii)}] $\cG (V)$ has no $3$-cycles.
  \item[{\rm (iii)}] Either $V$ has a composition series of length less than or equal to $2,$ or $V$ has a composition series of length $3$ and it has a unique maximal $R$-submodule.
\end{description}
\end{pro}
\begin{proof} (i)$\implies $(ii): Obvious.

(ii)$\implies $(iii): Suppose that $\cG (V)$ has no $3$-cycles. It follows from \ref{3.1} that $V$ has a composition series of length less than or equal to $3.$ Assume that the composition length of $V$ is $3.$ We want to show that $V$ has a unique maximal submodule. Suppose for a moment that $V$ has two maximal submodules $J_1$ and $J_2.$ It is clear that $J_1,J_2,$ and $J_1\cap J_2$ are distinct. We see that $J_1\cap J_2=0,$ because, otherwise, $(J_1,J_1\cap J_2,J_2,J_1)$ is a $3$-cycle. As $J_1\cap J_2=0,$ the natural map $$V\rightarrow V/J_1\times V/J_2$$ is an injective $R$-module homomorphism, which must be surjective because $J_1+J_2=M.$ But then, $V$ is isomorphic to the direct sum of the simple modules $V/J_1$ and $V/J_2,$ implying that the composition length of $V$ is $2.$

(iii)$\implies $(i): If $V$ has a composition series of length less than or equal to $2,$ then any nonzero submodule of $V$ is simple, so that $\cG (V)$ has no edges.

Suppose now that the composition length of $V$ is $3,$ and that $V$ has a unique maximal submodule $J.$ Then, every proper submodule of $V$ is inside $J.$ As the composition length of $J$ is $2,$ every proper submodule of $V$ different from $J$ is simple. Consequently, $\cG (V)$ is the star $\K _{1,c },$ where $J$ is adjacent to all the other vertices (= all the proper submodules of $J$), and there are no edges between the vertices other than $J,$ and $c$ denotes the cardinality of the set of all the proper submodules of $J.$ Hence, we have observed in both cases that there are no cycles.
\end{proof}

A graph $\cG $ is called bipartite if the vertex set of $\cG $ can be written as a disjoint union of two (possibly empty) sets $V_1$ and $V_2$ such that any edge in $\cG $ connects a vertex in $V_1$ with a vertex in $V_2.$ See, for instance, \cite{Ore}.

\begin{cor}
\label{3.3} $\cG (V)$ is bipartite if and only if either $V$ has a composition series of length less than or equal to $2,$ or $V$ has a composition series of length $3$ and it has a unique maximal $R$-submodule.
\end{cor}
\begin{proof} Suppose that $\cG (V)$ is bipartite. Then, it is clear and well known that $\cG (V)$ has no odd cycles (see, for instance, \cite[Theorem 7.1.1]{Ore}). The result now follows from \ref{3.2}. The converse follows from the last part of the proof of \ref{3.2}.
\end{proof}

The results \ref{3.2} and \ref{3.3} generalize the related results of \cite{Chak}.

\section{Domination}
In this section we determine the domination number of $\cG(V).$ This is already done in the special cases where $R$ is a commutative ring and $V=R,$ and where $R$ is a finite field and $V$ is a finite dimensional vector space over $R,$ see \cite{Jaf} and the references there. Here, without using the results in these special cases, we determine the domination number of $\cG(V)$ in the most general case where the ring $R$ and the $R$-module $V$ are arbitrary.

We sometimes use the following result of \cite{Fla} which describes the submodules of the direct sum of two modules. It is a simple modification of a result for groups known as Goursat's lemma.

\begin{lem}{\rm \cite[Theorem 3.1]{Fla}}
\label{4.1}
Let $U$ and $W$ be $R$-modules. Then:
\begin{description}
  \item[{\rm (1)}] There is a bijective map $\Psi $ from the set of all $R$-submodules of the direct sum $U\times W$ to the set of all quintuples $(U_1,U_2,\theta ,W_2,W_1)$ where $U_2\leq U_1\leq U$ and $W_2\leq W_1\leq W$ and $\theta :U_1/U_2\rightarrow W_1/W_2$ is an isomorphism of $R$-modules.
  \item[{\rm (2)}] $\Psi $ sends an $R$-submodule $M$ of $U\times W$ to the quintuple $$\big (p_1(M),k_1(M),\theta _M,k_2(M),p_2(M)\big )$$
      where $$p_1(M)=\{u\in U:\exists w\in W,(u,w)\in M \},{\ }{\ }k_1(M)=\{u\in U:(u,0)\in M \},$$
      $$p_2(M)=\{w\in W:\exists u\in U,(u,w)\in M \},{\ }{\ }k_2(M)=\{w\in W:(0,w)\in M \},$$
      and, for any $(u,w)\in p_1(M)\times p_2(M),$ the isomorphism $\theta _M$ sends $u+k_1(M)$ to $w+k_2(M)$ if and only if $(u,w)\in M.$ Furthermore, the three $R$-modules
      $$p_1(M)\big{/}k_1(M),{\ }{\ }{\ }p_2(M)\big{/}k_2(M),{\ }{\ }{\ }and{\ }{\ }{\ }M\big{/}\big (k_1(M)\times k_2(M)\big )$$
      are isomorphic.
  \item[{\rm (3)}] For any quintuple $(U_1,U_2,\theta ,W_2,W_1)$ where $U_2\leq U_1\leq U$ and $W_2\leq W_1\leq W$ and $\theta :U_1/U_2\rightarrow W_1/W_2$ is an isomorphism of $R$-modules, the inverse of $\Psi $ sends $$(U_1,U_2,\theta ,W_2,W_1)$$ to the module
      $$\{(u,w)\in U_1\times W_1:\theta (u+U_2)=w+W_2\}.$$
\end{description}
\end{lem}

Let $U$ and $W$ be $R$-submodules of an $R$-module $V$ such that $V=U\oplus W.$ Although \ref{4.1} is stated for the external direct sum of two modules, it also describes the submodules of the internal direct sum $U\oplus W$ because the $R$-modules $U\oplus W$ and $U\times W$ are isomorphic. Indeed, \ref{4.1} is stated in \cite{Fla} for an internal direct sum of two submodules of a module.

Let $\cG $ be a graph. By a dominating set for $\cG$ we mean a subset $D$ of the vertex set of $\cG $ such that every vertex not in $D$ is joined to at least one vertex in $D$ by some edge. A dominating set $D$ is called a minimal dominating set if $D'$ is not a dominating set for any subset $D'$ of $D$ with $D'\neq D.$ The domination number of $\cG $ is the smallest of the  cardinalities of the minimal dominating sets for $\cG .$ See, for instance, \cite{Ore}. An infinite graph may not have a minimal dominating set, in which case, the domination number is not defined. However, we observe in this section that the intersection graph of any module has a minimal dominating set.

In our case, a set $\cS$ of proper $R$-submodules of $V$ is a dominating set for $\cG (V)$ if and only if for any proper $R$-submodule $U$ of $V$ there is a $W$ in $\cS $ such that $U\cap W\neq 0.$

The existence of a maximal element mentioned in the following result is guaranteed by Zorn's lemma.
\begin{lem}
\label{4.2} Given any nonzero $R$-submodule $U$ of $V$ which is not simple, take a maximal element $W$ of the set  $\{X\leq V:X\cap U= 0\}$, a poset with respect to the subset or equal to relation. If $\cA $ is a dominating set for $\cG(U),$ then the set $\{T+W:T\in \cA \}$ is a dominating set for $\cG (V).$
\end{lem}
\begin{proof} Take a proper $R$-submodule $Y$ of $V.$ We want to show that $Y\cap (T+W)\neq 0$ for some $T\in \cA .$ We have three possibilities for the intersection $Y\cap U.$ It is either a proper submodule of $U,$ or $U,$ or else $0.$

If $Y \cap U$ is a proper submodule of $U,$ then there is a $T_1\in \cA$ such that $(Y\cap U)\cap T_1\neq 0$ as $\cA $ is a dominating set for $\cG (U).$ As $(Y\cap U)\cap T_1\subseteq Y\cap (T_1+W),$ this implies that $Y\cap (T_1+W)\neq 0.$

If $Y\cap U=U,$ then $Y$ contains $U$ and so contains every element of $\cA .$ Thus, for any $T_2\in \cA ,$ we have that $0\neq T_2\subseteq Y\cap (T_2+W).$

Suppose now that $Y\cap U=0.$ If $Y\leq W$ then $Y\cap (T_3+W)\neq 0$ for any $T_3\in \cA .$ So we assume that $Y$ is not contained in $W,$ implying that $W$ is a proper subset of $Y+W.$ The maximality of $W$ gives that $U\cap (Y+W)\neq 0.$ As $U$ is not simple, there is a nonzero element $u$ in $U\cap (Y+W)$ such that $0\neq Ru\neq U.$ (If $U\cap (Y+W)$ is not equal to $U,$ then $u$ may be any element of $U\cap (Y+W)$). Since $\cA$ is a dominating set for $\cG (U),$ there is a $T_4\in \cA $ such that $Ru\cap T_4\neq 0.$ Hence, there is an $r\in R$ such that $ru$ is a nonzero element of $T_4.$ As $ru$ is in $Y+W$ (because $u$ is in it), we write $ru=y+w$ for some $y\in Y$ and $w\in W.$ Note that both of $y$ and $w$ are nonzero, because $Y\cap U= 0$ and $Y\cap W= 0.$ Then, $y=ru-w\in (T_4+W)$ is a nonzero element of $Y\cap (T_4+W).$
\end{proof}

Assuming that the domination numbers of $\cG (U)$ and $\cG (V)$ are defined, the above result implies that the domination number of $\cG (U)$ is greater than or equal to the domination number of $\cG (V)$ for any non zero and non simple $R$-submodule $U$ of $V.$ In the next result we see a case in which a converse inequality holds.

\begin{lem}
\label{4.3}  Let $V$ be a semisimple $R$-module containing a direct sum of three simple modules, and let $\cC$ be the set of all $R$-submodules of $V$ that can be written as a direct sum of two simple $R$-modules. Let $\cA $ be a minimal dominating set for $\cG (V).$ Then, there is a $U\in \cC$ such that $U\cap A\neq U$ for all $A\in \cA .$ Moreover, for any such $U,$ the nonzero elements of the set $\{T\cap U:T\in \cA \}$ forms a dominating set for $\cG (U).$
\end{lem}
\begin{proof} We first show the existence of such an element $U$ of $\cC .$ Suppose the contrary. Take any element $W$ of $\cC .$ By the assumption there is an $A\in \cA$ such that $W\cap A= W.$ We will show that $\cA -\{ A\}$ forms a dominating set for $\cG (V),$ which contradicts to the minimality of $\cA .$

Let $X$ be a proper $R$-submodule of $V.$ We need to show that the intersection of $X$ with an element of $\cA -\{A\}$ is nonzero. As $\cA $ is a dominating set for $\cG (V),$ the intersection of $X$ with an element of $\cA $ is nonzero. So, we may assume that $X\cap A\neq 0.$ As $V$ is semisimple, $A$ has a complement $B$ in $V$ so that we write $V=A\oplus B.$ Note that $B$ is nonzero as $A\in \cA .$ Again by using the semisimplicity of $V,$ we choose a simple $R$-submodule $S_1$ of $X\cap A$ and a simple $R$-submodule $S_2$ of $B.$ We then let $W'=S_1+S_2.$ As $A\cap B=0,$ the simple modules $S_1$ and $S_2$ are different so that $W'\in \cC .$ Since we assumed the contrary, there is a $A'\in \cA $ such that $W'\cap A'=W'.$ Now, $S_2$ is in $A'$ but not in $A,$ proving that $A'\neq A,$ so $A'\in \cA -\{A\}.$ Moreover, $X\cap A'\neq 0$ because it contains $S_1.$

Having proved the existence of such an element $U$ of $\cC,$ we see that what is left is trivial to justify.
\end{proof}

We now state the main result of this section. For a graph $\cG $ we denote by $\gamma (\cG )$ the domination number of $\cG .$ If $V$ is a simple $R$-module then, of course, $\cG (V)$ has no vertices so that $\gamma (\cG (V))=0.$
\begin{thm}
\label{4.4} Suppose that $V$ is not a simple $R$-module. Then:
\begin{description}
  \item {\rm (1)} $\gamma (\cG (V))=1$ if and only if $V$ is not a semisimple $R$-module.
  \item {\rm (2)} $\gamma (\cG (V))=2$ if and only if $V$ is a semisimple $R$-module having at least two nonisomorphic simple $R$-submodules.
  \item {\rm (3)} If $V$ is a semisimple $R$-module all of whose simple $R$-submodules are isomorphic, then $\gamma (\cG (V))=|\End _R(S)|+1$ where $S$ is a simple $R$-submodule of $V.$
\end{description}
\end{thm}
\begin{proof} (1) This is trivial, because it is well known that an application of Zorn's lemma shows that: $V$ is not a semisimple module if and only if $V$ has a proper essential submodule (that is, a submodule intersecting every nonzero submodule of $V$ nontrivially). See, for instance, \cite[Corollary 5.9]{Warf}.

(2) Suppose that $V$ is a semisimple module having at least two nonisomorphic simple submodules, say $S_1$ and $S_2$. We here show that $\gamma (\cG (V))=2.$ The converse implication will follow by the virtue of parts (1) and (3), because Schur's lemma implies that $\End _R(S)$ is a division ring and so $|\End _R(S)|+1\geq 3.$

Let $V_1$ be the sum of all simple $R$-submodules of $V$ isomorphic to $S_1,$ and let $V_2$ be a complement of $V_1$ in $V$ so that $V=V_1\oplus V_2.$ By their constructions, $V_1$ and $V_2$ have no isomorphic sections. Thus, it follows from \ref{4.1} that the submodules of $V$ are of the form $X\oplus Y$ where $X\leq V_1$ and $Y\leq V_2.$ Now, it is obvious that the set $\{V_1,V_2\}$ forms a dominating set for $\cG (V).$ So, $\gamma (\cG (V))\leq 2,$ but it cannot be $1$ by part (1).

(3) Firstly, we show that $\cG (V)$ has a minimal dominating set of cardinality $|\End _R(S)|+1$ where $S$ is a simple submodule of $V.$ Take any submodule $U$ of $V$ that is a direct sum of two simple modules. Let $\cA $ be a dominating set for $\cG(U).$ As the proper submodules of $U$ are all simple, $\cA $ consists of all the proper submodules of $U.$ Using \ref{4.1} we see that the submodules of $S\times S$ are precisely
$$0\times 0,{\ }{\ }{\ }S\times 0,{\ }{\ }{\ }0\times S,{\ }{\ }{\ }S\times S,{\ }{\ }{\ }\{(s,\alpha (s)):s\in S\},$$
where $\alpha $ is ranging in the set of all the automorphisms of $S.$ As $U$ is isomorphic to $S\times S,$ we count the number of the proper submodules of $U$ as $|\Aut (S)|+2.$ Since $S$ is simple, Schur's lemma implies that $|\cA |=|\End _R(S)|+1.$ Since $V$ is semisimple, $U$ has a complement $W$ in $V.$ Now it follows from \ref{4.2} that the set $\cB =\{T+W:T\in \cA \}$ is a dominating set for $\cG(V).$ As $U\cap W=0,$ we see that the modules $T+W$ are all distinct where $T$ is ranging in $\cA ,$ so that $|\cB |=|\cA |.$ So far we have proved that $\cB $ is a dominating set for $\cG (V)$ of desired cardinality. To see that $\cB $ is a minimal dominating set, we simply observe that, for any $T\in \cA ,$ the only element of $\cB $ adjacent to $T$ is $T+W,$ because of $U\cap W=0.$

To finish the proof, we show that any minimal dominating set for $\cG(V)$ has cardinality greater than or equal to $|\End _R(S)|+1.$ Indeed, let $\cD $ be a minimal dominating set for $\cG (V).$ Using \ref{4.3} we deduce that $V$ has a submodule $U'$ isomorphic to $S\times S$ such that $U'$ has a dominating set of cardinality less than or equal to $|\cD |.$ As in the previous paragraph, since $U'$ is isomorphic to $S\times S,$ we see that any dominating set for $\cG (U')$ must contain every proper submodule of $U'$ and that $U'$ has $|\End _R(S)|+1$ proper submodules.
\end{proof}
The proof of the above result shows that the intersection graph of any module has a minimal dominating set.

\section{Finiteness}
In this section we investigate the modules $V$ for which $\cG (V)$ has finitely many vertices.
\begin{lem}
\label{5.1} Let $S$ be a simple $R$-module and $n$ be a natural number. Let $S_n$ denote the direct sum
$$\underbrace{S\times S\times \cdot \cdot \cdot \times S}_{n{\ }times },$$ which is a semisimple $R$-module of composition length $n.$ If $w _n$ denotes the number of maximal $R$-submodules of $S_n,$ then $\ds w _n=\frac{d^n-1}{d-1}$ where $d=|\End _R(S)|.$ In particular, for $n\geq 2,$ $w _n$ is finite if and only if $d$ is finite.
\end{lem}
\begin{proof} We write $S_{n+1}=S_n\times S.$ It follows from \ref{4.1} that there is a bijection from the set of all $R$-submodules of $S_{n+1}$ to the set of all quintuples $(U_1,U_2,\theta ,W_2,W_1)$ where $U_2\leq U_1\leq S_n$ and $W_2\leq W_1\leq S$ and $\theta :U_1/U_2\rightarrow W_1/W_2$ is an isomorphism of $R$-modules. Since $S$ is simple, the quintuples are $$(U,U,1,0,0),{\ }{\ }{\ }(U,U,1,S,S),{\ }{\ }{\ }(U_1,U_2,\theta ,0,S)$$ where $U\leq S_n\geq U_1\geq U_2,$ and $1$ is the unique isomorphism between the zero modules $U/U,$ $0/0,$ and $S/S,$ and $\theta $ is an isomorphism from $U_1/U_2$ to $S/0,$ and $U_2$ is a maximal submodule of $U_1.$ Using \ref{4.1} again we now want to determine which quintuples correspond to the maximal submodules of $S_{n+1}.$ It is clear (especially from the isomorphisms of three modules in part (2) of) \ref{4.1} that the quintuples corresponding to the maximal submodules of $S_{n+1}$ are precisely
$$(S_n,S_n,1,0,0),{\ }{\ }{\ }(U,U,1,S,S),{\ }{\ }{\ }(S_n,U_2,\theta ,0,S)$$ where $U$ and $U_2$ are maximal submodule of $S_n,$ and $1$ and $\theta $ are isomorphisms as above. The number of quintuples of the form $(U,U,1,S,S)$ where $U$ is a maximal submodule of $S_n$ is $w_n.$ The number of quintuples of the form $(S_n,U_2,\theta ,0,S)$ where $U_2$ is a maximal submodule of $S_n$ and $\theta :S_n/U_2\rightarrow S/0$ is an isomorphisms is $|\Aut _R(S)|w_n$ because, for any maximal submodule $U_2$ of $S_n,$ the number of isomorphisms from $S_n/U_2$ to $S/0$ is equal to the number of automorphisms of $S.$ As $S$ is simple, Schur's lemma implies that $|\Aut _R(S)|=d-1.$ Consequently, the number of quintuples corresponding to the maximal submodules of $S_{n+1}$ is $1+w_n+(d-1)w_n.$ Hence, we obtain the recurrence relation $$w_{n+1}=1+dw_n,$$ which can be easily solved to obtain the result.
\end{proof}

Suppose that we have an $n$ dimensional vector space over a finite field with $d$ elements. One may count the number of its maximal subspaces more easily to be $\ds \frac{d^n-1}{d-1},$ the same number in \ref{5.1}. We see in the next section that this is not a coincidence.

\begin{pro}
\label{5.2} Let $V$ be a semisimple $R$-module isomorphic to $$n_1T_1\oplus n_2T_2\oplus \cdot \cdot \cdot \oplus n_rT_r$$ where $r$ is a natural number, $T_1,T_2,...,T_r$ are mutually nonisomorphic simple $R$-modules, and $n_1,n_2,...,n_r$ are natural numbers. Then, the number of maximal submodules of $V$ is
$$\sum _{i=1}^r\frac{d_i^{n_i}-1}{d_i-1},$$ where $d_i=|\End _R(T_i)|.$ In particular, the number of maximal submodules of $V$ is finite if and only if $d_k$ is finite for any $k\in \{1,2,...,r\}$ with $n_k\geq 2.$
\end{pro}
\begin{proof} Let $W=V_1\times V_2\times \cdot \cdot \cdot \times V_r$ where $V_i=n_iT_i,$ so that $V\cong W.$ It follows from \ref{4.1} that the submodules of $W$ are of the form $A_1\times A_2\times \cdot \cdot \cdot \times A_r$ where $A_i\leq V_i$ for each $i.$ Moreover, it is clear that $A_1\times A_2\times \cdot \cdot \cdot \times A_r$ is a maximal submodule of $W$ if and only if there is an index $k$ such that $A_k$ is a maximal submodule of $V_k$ and $A_i=V_i$ for any $i$ different from $k.$ Therefore, the result follows from \ref{5.1}.
\end{proof}

In the following we characterize $R$-modules $V$ for which $\cG (V)$ has finitely many vertices.
\begin{thm}
\label{5.3} The number of $R$-submodules of $V$ is finite if and only if the following conditions hold:
\begin{description}
  \item {\rm (a)} $V$ has a (finite) composition series.
  \item {\rm (b)} If $V$ has a semisimple section in which the multiplicity of a simple $R$-module $S$ as a composition factor is greater than one, then $|\End _R(S)|$ is finite.
\end{description}
\end{thm}
\begin{proof} Suppose that the number of submodules of $V$ is finite. It is clear that $V$ can not contain an infinite ascending or descending chain of submodules so that $V$ is both Noetherian and Artinian, and hence it has a (finite) composition series. Suppose that there are submodules $X\leq Y$ of $V$ such that $Y/X$ is semisimple and that $Y/X$ has a submodule isomorphic to the direct sum $S\times S$ for some simple module $S.$ It follows from \ref{5.1} that the number of submodules of $Y/X$ is greater than $1+|\End _R(S)|.$ As the number of submodules of $V$ is greater than or equal to the number of submodules of $Y/X,$ it follows that $\End _R(S)$ is finite.

Conversely, suppose that $V$ is a module satisfying the conditions (a) and (b). Suppose for a moment that $V$ has infinitely many submodules. Because of $(a),$ every proper submodule of $V$ is contained in a maximal submodule of $V.$ Hence, either $V$ has infinitely many maximal submodules or there is a maximal submodule $V_1$ of $V$ such that $V_1$ has infinitely many submodules. The first possibility does not occur here, because the number of maximal submodules of $V$ is equal to the number of maximal submodules of the semisimple module $V/\Jac (V),$ which is finite because of (b) and \ref{5.2}. Therefore, $V_1$ has infinitely many submodules, and applying the same argument to $V_1$ we find a maximal submodule $V_2$ of $V_1$ such that $V_2$ has infinitely many submodules. Applying the same argument to $V_2$ we obtain $V_3.$ Continuing in this way we find an infinite descending chain $V>V_1>V_2>V_3>\cdot \cdot \cdot >\cdot \cdot \cdot .$ This contradicts with (a).
\end{proof}

\begin{rem}
\label{5.4} Suppose that $\cG (V)$ has an edge. Then, $\cG (V)$ has finitely many vertices if and only if $\cG (V)$ has finitely many edges.
\end{rem}
\begin{proof} Suppose that there are finitely many vertices. Since $\cG (V)$ has no multiple edges, it must have finitely many edges. Conversely, suppose that there are finitely many edges. As $\cG (V)$ has an edge, it follows form \ref{2.4} that $\cG (V)$ has no isolated vertices. As an edge determines two vertices, we see that $\cG (V)$ must have finitely many vertices.
\end{proof}

\section{Counting}
In this section we want to obtain more counting results about the submodules of a semisimple module.

We begin with an elementary observation. For the lack of an appropriate reference, we include a short justification.
\begin{rem}
\label{6.1}
\begin{description}
  \item[{\rm (1)}] Suppose that $V$ is a semisimple $R$-module having a (finite) composition series. Let $\cP $ be the poset of all the right ideals of $\End _R(V)$ and $\cQ $ be the poset of all the $R$-submodules of $V.$ For any right ideal $J$ of $\End _R(V)$ we let
      $$\phi _{*}(J)=\{\sum _{k=1}^mf_k(v_k):m\in \mathbb{N},f_k\in J,v_k\in V\},$$ the $R$-submodule of $V$ generated by the set $\{f(v):f\in J, v\in V\}.$ For any $R$-submodule $W$ of $V$ we let
      $$\phi ^{*}(W)=\{f\in \End _R(V):f(V)\subseteq W\}.$$ Then, the maps $\phi _{*}:\cP \rightarrow \cQ$ and $\phi ^{*}:\cQ \rightarrow \cP$ are mutually inverse poset isomorphisms.
  \item[{\rm (2)}] Let $A$ be a unital ring and $n$ be a natural number. Let $\M _n(A)$ be the ring of all the $n\times n$ matrices with entries in $A,$ and $\C _n(A)$ be the set of all the $n\times 1$ matrices with entries in $A.$ Identifying $A$ with $\M _1(A)$ we may regard $\C _n(A)$ as a right $A$-module with the module action given by matrix multiplication. Let $\cR $ be the poset of all the right ideals of $\M _n(A)$ and $\cT $ be the poset of all the right $A$-submodules of $\C _n(A).$ For any right ideal $I$ of $\M _n(A)$ we let
      $$\varphi _{*}(I)=\{Xe_{11}:X\in I\},$$ and for any right $A$-submodule $U$ of $\C _n(A)$ we let
      $$\varphi ^{*}(U)=\{X\in \M _n(A):Xe_{i1}\in U{\ }for{\ }any{\ }i=1,2,...,n\},$$ where $e_{i1}$ denote the $n\times 1$ matrix units in $\C _n(A)$ so that $Xe_{i1}$ is the $i$th column of $X.$  Then, the maps $\varphi _{*}:\cR \rightarrow \cT$ and $\varphi ^{*}:\cT \rightarrow \cR$ are mutually inverse poset isomorphisms.
\end{description}

\end{rem}
\begin{proof} (1) It is clear from their definitions that $J\leq \phi ^{*}\phi _{*}(J)$ and $\phi _{*}\phi ^{*}(W)\leq W$ for any right ideal $J$ of $\End _R(V)$ and any submodule $W$ of $V.$ We will observe that the reverse inclusions are also true, implying that $\phi _{*}$ and $\phi _{*}$ are mutually inverse bijections.

As $V$ is semisimple, there is a projection $\pi \in \End _R(V)$ onto $W.$ Then, $\pi \in \phi ^{*}(W)$ and so $W=\pi (V)\leq  \phi _{*}\phi ^{*}(W).$ Thus, $\phi _{*}\phi ^{*}$ is the identity.

As $V$ is a semisimple $R$-module having a (finite) composition series, $\End _R(V)$ is a direct sum of finitely many matrix rings over division rings so that it is a semisimple ring. Thus, a right ideal $J$ of $\End _R(V)$ must be principle generated by an idempotent $f_J\in \End _R(V),$ so that $J=f_J\End _R(V).$ We then see that $\phi _{*}(J)=f_J(V).$ Thus, if $f\in \phi ^{*}\phi _{*}(J)$ then $f(V)\subseteq \phi _{*}(J)=f_J(V).$ As $f_J$ is an idempotent, the last containment shows that $f_Jf=f,$ proving that $f\in J.$ Thus, $\phi ^{*}\phi _{*}$ is the identity.

Finally, it is obvious from their definitions that both of $\phi _{*}$ and $\phi ^{*}$ preserve orders. Hence, they must be poset isomorphisms.

(2) Straightforward.
\end{proof}

When $V$ is not a semisimple $R$-module, the maps in part (1) of \ref{6.1} are no longer bijections, however they form an order preserving Galois connection. See, for instance, \cite{Khu} and the references there.

For nonnegative integers $n\geq m$ and for a natural number $d\geq 2,$ we let
$$\binom{n}{m}_d=\prod _{k=1}^m\frac{d^n-d^{k-1}}{d^m-d^{k-1}}=\frac{(d^n-1)(d^n-d)\cdot \cdot \cdot (d^n-d^{m-1})}{(d^m-1)(d^m-d)\cdot \cdot \cdot (d^m-d^{m-1})},$$ which is called the $d$-ary Gaussian binomial coefficient.

\begin{lem}
\label{6.2} Let $S$ be a simple $R$-module and $n$ be a natural number. Let $S_n$ denote the direct sum
$$\underbrace{S\times S\times \cdot \cdot \cdot \times S}_{n{\ }times },$$ which is a semisimple $R$-module of composition length $n.$ For any nonnegative integer $i\leq n,$ let $\mu _i$ denote the number of $R$-submodules of $S_n$ of composition length $i.$ Let $j$ and $m$ be nonnegative integers less than or equal to $n.$ Then:
\begin{description}
  \item[{\rm (1)}] The poset of all the $R$-submodules of $S_n$ is isomorphic to the poset of all the subspaces of an $n$ dimensional vector space over the division ring $\End _R(S).$ Under this isomorphism an $R$-submodule of $S_n$ of composition length $i$ corresponds to an $i$ dimensional subspace of the vector space.
  \item[{\rm (2)}] $$\mu _i=\binom{n}{i}_d$$ where $d=|\End _R(S)|.$ In particular, for $n\geq 2,$ $\mu _i$ is finite if and only if $d$ is finite.
  \item[{\rm (3)}] Let $U$ be an $R$-submodule of $S_n$ of composition length $j.$ Then, the number of $R$-submodules $W$ of $S_n$ of composition length $i$ such that $U\cap W$ is of composition length $m$ is $$d^{(i-m)(j-m)}\binom{n-j}{i-m}_d\binom{j}{m}_d.$$
      In particular, the number of complements of $U$ in $W$ is $d^{(n-j)j}.$
\end{description}
\end{lem}
\begin{proof} (1) Let $A=\End _R(S)$ which is a division ring by Schur's lemma. We use the notations in \ref{6.1}. Observe that $\cT $ may be identified with the poset of all the right $A$-submodules of $$\underbrace{A\times A\times \cdot \cdot \cdot \times A}_{n{\ }times},$$ an $n$ dimensional vector space over the division ring $A$ on which scalar acts from right. Moreover, as the ring $\End _R(S_n)$ is isomorphic to $\M _n(A),$ we see that the posets of all the right ideals of the rings $\End _R(S_n)$ and $\M _n(A)$ are isomorphic. The result now follows by the virtue of \ref{6.1}.

(2) By part (1), we see that $\mu _i$ is equal to the number of $i$ dimensional subspaces of an $n$-dimensional vector space over a division ring. The number of such subspaces is easy to find. See, for instance, page 269 of \cite[9.3.2. Lemma]{Bro}, implying the result.

(3) This again follows from part (1) and \cite[9.3.2. Lemma]{Bro}.
\end{proof}

The numbers in the last two parts of \ref{6.2} may also be calculated by using \ref{4.1}. However, reducing to the vector spaces simplifies calculations of these numbers. Moreover, using part (1) of \ref{6.2} and \ref{4.1}, one may generalize known results on the poset of subspaces of a vector space to results on the poset of submodules of a semisimple module. We here give some examples of this.

Part (1) of \ref{6.2} and \ref{4.1} implies

\begin{rem}
\label{6.3} Let $V$ be a semisimple $R$-module having a (finite) composition series. Then, the lattice of $R$-submodules of $V$ is isomorphic to a direct product of finitely many lattices of subspaces of finite dimensional vector spaces over division rings. More precisely, if $$V\cong n_1T_1\oplus n_2T_2\oplus \cdot \cdot \cdot \oplus n_rT_r,$$ where $r$ is a natural number, $T_1,T_2,...,T_r$ are mutually nonisomorphic simple $R$-modules, and $n_1,n_2,...,n_r$ are natural numbers, then we have the lattice isomorphism
$$\cL (V)\cong \cL ({U_1})\times \cL (U_2)\times \cdot \cdot \cdot \times \cL (U_r),$$ where $U_i$ is an $n_i$ dimensional vector space over the division ring $\End _R(T_i)$ on which scalars act from right, and $\cL (U_i)$ denotes the lattice of subspaces of $U_i,$ and $\cL (V)$ denotes the lattice of $R$-submodules of $V.$
\end{rem}

\begin{pro}
\label{6.4} Let $V$ be a semisimple $R$-module isomorphic to $$n_1T_1\oplus n_2T_2\oplus \cdot \cdot \cdot \oplus n_rT_r$$ where $r$ is a natural number, $T_1,T_2,...,T_r$ are mutually nonisomorphic simple $R$-modules, and $n_1,n_2,...,n_r$ are natural numbers.
\begin{description}
 \item[{\rm (1)}] For any nonnegative integer $i\leq n_1+n_2+\cdot \cdot \cdot +n_r,$ the number of $R$-submodules of $V$ of composition length $i$ is
     $$\sum _{(i_1,i_2,...,i_r)}\prod _{k=1}^r\binom{n_k}{i_k}_{d_k}$$ where the $r$-tuples $(i_1,i_2,...,i_r)$ are ranging in the set of all nonnegative integer solutions of the linear equation $x_1+x_2+\cdot \cdot \cdot +x_r=i$ satisfying $x_k\leq n_k$ for all $k,$ and where $d_k=|\End _R(T_k)|$ for all $k.$
 \item[{\rm (2)}] Let $U$ be an $R$-submodule of $V.$ Then, the number of complements of $U$ in $V$ is
 $$\prod _{i=1}^rd_k^{(n_k-m_k)m_k}$$ where, for each $k,$ $d_k=|\End _R(T_k)|,$ and $m_k$ is the multiplicity of $T_k$ as a composition factor of $U,$ so that $U$ is isomorphic to $m_1T_1\oplus m_2T_2\oplus \cdot \cdot \cdot \oplus m_rT_r.$
\end{description}
\end{pro}
\begin{proof} Let $W=V_1\times V_2\times \cdot \cdot \cdot \times V_r$ where $V_i=n_iT_i,$ so that $V\cong W.$ It follows from \ref{4.1} that the submodules of $W$ are of the form $A_1\times A_2\times \cdot \cdot \cdot \times A_r$ where $A_i\leq V_i$ for each $i.$ The result follows from \ref{6.2}.
\end{proof}

The number in part (2) of \ref{6.4} is equal to the cardinality of the set $\Hom _R(V/U,U).$ Indeed, a more general result is true. Suppose $W$ is an $R$-module which is not necessarily semisimple and $A$ be an $R$-submodule of $W.$ If $A$ has a complement in $W,$ then it is easy to see by applying \ref{4.1} that there is a bijection between the set of complements of $A$ in $W$ and the hom set $\Hom _R(W/A,A).$

We have the following obvious consequence of \ref{6.4}.
\begin{cor}
\label{6.5} Let $V$ be a semisimple $R$-module of composition length $n$ where $n$ is a natural number. For any nonnegative integer $i\leq n,$ let $\mu _i$ be the number of $R$-submodules of $V$ of composition length $i.$ Let $U$ be an $R$-submodule of $V.$ Then:
\begin{description}
 \item[{\rm (1)}] For any nonnegative integer $i\leq n,$ the numbers $\mu _i$ and $\mu _{n-i}$ are both infinite or they are equal finite numbers.
 \item[{\rm (2)}] For any complement $W$ of $U$ in $V,$ the number of complements of $U$ in $V$ and the number of complements of $W$ in $V$ are both infinite or they are equal finite numbers.
 \item[{\rm (3)}] For any two isomorphic $R$-submodules $X$ and $Y$ of $V,$ the number of complements of $X$ in $V$ and the number of complements of $Y$ in $V$ are both infinite or they are equal finite numbers.
\end{description}
\end{cor}

Let $k$ be a natural number and $L$ be a finite modular lattice. It is shown in \cite{Dil} that the number of elements of $L$ covered by precisely $k$ elements is equal to the number of elements of $L$ covering precisely $k$ elements. The submodule lattice of (any module) $V$ is modular, and if we assume further that $V$ is semisimple, then, for $k=1,$ this result shows that the number of simple submodules of $V$ is equal to the number of maximal submodules of $V,$ i.e., $\mu _1=\mu _{n-1}$ with the notations of \ref{6.5}. On the other hand, let $D$ be a finite modular complemented lattice. By the length of an element $x$ in $D$ we mean the length of the largest chain between the smallest element of $D$ and $x.$ Let the length of the greatest element of $D$ be $n.$ Then, \ref{6.5} prompts us to ask: Is it true that the number of elements of $D$ of length $k$ is equal to the number of elements of $D$ of length $n-k$? An affirmative answer implies part (1) of \ref{6.5}. One may ask a similar question for $D$ suggested by part (2) of \ref{6.5}.

\begin{rem}{\rm \cite{Cla}}
\label{6.6} Let $U$ be an $n$ dimensional vector space over a finite field $\mathbb{F}$ with $q$ elements. Let $\cS (U)$ denote the set of all subspaces of $U,$ and let $\cS _i(U)$ denote the set of all $i$ dimensional subspaces of $U$ where $0\leq i\leq n.$ Then:
\begin{description}
 \item[{\rm (1)}] There is a bijection $\phi :\cS (U)\rightarrow \cS (U)$ such that $X\cap \phi (X)=0$ and $X+\phi(X)=U$ for any $X\in \cS (U).$ In particular, for any $j$ with $0\leq j\leq n,$ the restriction of $\phi $ to $\cS _j(U)$ gives a bijection $\phi _j:\cS _j(U)\rightarrow \cS _{n-j}(U).$ Moreover, if $n$ is odd, then $\phi \circ \phi $ is the identity on $\cS (U).$ More to the point, if $n=2k$ is even and $q$ is odd, then $\phi _{k}\circ \phi _{k}$ is the identity on $\cS _k(U).$
 \item[{\rm (2)}] Suppose that $n=2k$ is even and $q$ is odd. Then, we may write $\cS _k(U)=\cA \cup \cB $ for some disjoint sets $\cA $ and $\cB $ satisfying the condition: there is a bijection $\alpha :\cA \rightarrow \cB $ such that $X\cap \phi (X)=0$ for any $X\in \cA.$
 \item[{\rm (3)}] Suppose that $n=2k$ is even and $q>2$ is even. Then, there is a $W\in \cS _k(U)$ such that we may write $\cS _k(U)-\{W\}=\cA \cup \cB $ for some disjoint sets $\cA $ and $\cB $ satisfying the condition: there is a bijection $\alpha :\cA \rightarrow \cB $ such that $X\cap \phi (X)=0$ for any $X\in \cA.$
\end{description}
\end{rem}

By using part (1) of \ref{6.2} and \ref{4.1} we lift the previous result to semisimple modules as follows, and we will use it in the next section where we study colorings.

\begin{lem}
\label{6.7} Let $V$ be a semisimple $R$-module isomorphic to $$n_1T_1\oplus n_2T_2\oplus \cdot \cdot \cdot \oplus n_rT_r$$ where $r$ is a natural number, $T_1,T_2,...,T_r$ are mutually nonisomorphic simple $R$-modules, and $n_1,n_2,...,n_r$ are natural numbers. Assume that, for any natural number $i$ with $1\leq i\leq r,$ if $n_i\geq 2$ then $d_i=|\End _R(T_i)|$ is finite. Let $\cS (V)$ denote the set of all $R$-submodules of $V,$ and let $\cS _j(V)$ denote the set of all $R$-submodules of $V$ of composition length $j$ where $j$ is a nonnegative integer. Let $n=n_1+n_2+\cdot \cdot \cdot n_r.$ Then:
\begin{description}
 \item[{\rm (1)}] There is a bijection $\phi :\cS (V)\rightarrow \cS (V)$ such that $X\cap \phi (X)=0$ and $X+\phi(X)=V$ for any $X\in \cS (V).$ In particular, for any $j$ with $0\leq j\leq n,$ the restriction of $\phi $ to $\cS _j(V)$ gives a bijection $\phi _j:\cS _j(V)\rightarrow \cS _{n-j}(V).$
 \item[{\rm (2)}] Suppose that $n=2k$ is even and at least one of $n_1,n_2,...,n_r$ is odd. Then, we may write $\cS _k(V)=\cA \cup \cB $ for some disjoint sets $\cA $ and $\cB $ satisfying the condition: there is a bijection $\alpha :\cA \rightarrow \cB $ such that $X\cap \phi (X)=0$ for any $X\in \cA.$
 \item[{\rm (3)}] Suppose that $n=2k$ is even and all of $n_1,n_2,...,n_r$ are even. If at least one of $d_1,d_2,...,d_r$ is odd, then the conclusion of part (2) is still true for $\cS _k(V).$
\end{description}
\end{lem}
\begin{proof} Let $\cS ({\ })$ denote the set of all submodules of its argument, and let $\cS _j({\ })$ denote the set of all submodules of its argument of composition length $j.$

(1) Let $V_i=n_iT_i.$ From part (1) of \ref{6.2} and \ref{6.6}, there is a bijection $\varphi _i:\cS (V_i)\rightarrow \cS (V_i)$ such that $X\cap \varphi _i(X)=0$ and $X+\varphi _i(X)=V_i$ for any $X\in \cS (V_i).$ Let $W=V_1\times V_2\times \cdot \cdot \cdot \times V_r$ so that $V\cong W.$ It follows from \ref{4.1} that
$$\cS (W)=\cS (V_1)\times \cS (V_2)\times \cdot \cdot \cdot \times \cS (V_r).$$ Then, $\varphi _1\times \varphi _2\times \cdot \cdot \cdot \times \varphi _r:\cS (W)\rightarrow \cS (W)$ is a bijection satisfying the required property for all $X\in \cS (W).$ The result follows because $V$ and $W$ are isomorphic modules.

(2) Without loss of generality assume that $n_1$ is odd. Write $W=V_1\times V'$ where $V_1=n_1T_1$ and $V'=n_2T_2\oplus \cdot \cdot \cdot \oplus n_rT_r$ so that $W\cong V.$ From \ref{4.1},
$$\cS _k(W)=\biguplus _i\cS _i(V_1)\times \cS _{k-i}(V'),$$ where $i$ is ranging in nonnegative integers such that $0\leq i\leq n_1$ and $0\leq k-i\leq 2k-n_1,$ (so that $\max(n_1-k,0)\leq i\leq \min(n_1,k_1)$), and where $\uplus$ denotes the disjoint union. To simplify the notations let $\cP _i=\cS _i(V_1)\times \cS _{k-i}(V').$ We have two possibilities: $k\geq n_1$ or $k<n_1.$ Assume first that $k\geq n_1$ so that $0\leq i\leq n_1.$ Letting $n_1=2s-1$ for some natural number $s,$ we now define
$$\cA= \cP _0\uplus \cP _1\uplus \cdot \cdot \cdot \uplus \cP _{s-1},{\ }{\ }{\ }\cB =\cP _{2s-1}\uplus \cP _{2s-2}\uplus \cdot \cdot \cdot \uplus \cP _s.$$ Thus, $\cS _k(W)=\cA \uplus \cB .$ Let $0\leq m\leq s-1$ be a natural number. It follows from part (1) that there are bijections $$\eta :\cS _m(V_1)\rightarrow \cS _{n_1-m}(V_1)$$ such that $X\cap \eta (X)=0$ and $X+\eta (X)=V_1$ for any $X\in \cS _m(V_1),$ and $$\theta :\cS _{k-m}(V')\rightarrow \cS _{(n-n_1)-(k-m)}(V')=\cS _{k-(n_1-m)}(V')$$ such that $Y\cap \theta (Y)=0$ and $Y+\theta (Y)=V'$ for any $Y\in \cS _{k-m}(V').$ Then, $$\eta \uplus \theta :\cP _m\rightarrow  \cP _{n_1-m}$$ is a bijection satisfying the required condition. The collection of the maps $\eta \uplus \theta $ for each $m$ gives a bijection $\cA \rightarrow \cB $ satisfying the required condition. The result follows because the posets $\cS _k(W)$ and $\cS _k(V)$ are isomorphic. The case in which $k<n_1$ may be handled similarly.

(3) Similar to the proof of part (2).
\end{proof}

\section{Coloring}
In this section we study the chromatic number of $\cG (V).$ The chromatic number of $\cG (V)$ is already studied in \cite{Jaf2} in the special case where $R$ is a finite field and $V$ is an odd dimensional vector space over $R.$ Therefore, some of our results generalize these special results.

Let $\cG $ be a graph. The chromatic number of $\cG $ is defined to be the smallest number of colors $\chi (\cG )$ needed to color the vertices of $\cG $ so that no two adjacent vertices share the same color. Whenever we mention about to color a graph or a coloring of a graph we assume that no two adjacent vertices share the same color. See, for instance, \cite{West} for finite graphs, and \cite{Ore} for infinite graphs. The smallest number of colors in the above definition may not make sense, when the graph has infinitely many vertices and it cannot be colored by finitely many colors.

We first understand when $\cG (V)$ can be colored by finitely many colors.

\begin{pro}
\label{7.1} $\cG (V)$ can be colored by finitely many colors if and only if the following conditions hold:
\begin{description}
 \item[{\rm (a)}] $V$ has a (finite) composition series.
 \item[{\rm (b)}] There is a natural number $n$ such that, for any simple $R$-submodule of $S$ of $V,$ the number of $R$-submodules of $V/S$ is less than or equal to $n.$
 \item[{\rm (c)}] $\cG (\Soc (V))$ can be colored by finitely many colors.
\end{description}
Moreover, in this case, let $\widetilde{\cG }$ be the subgraph of $\cG (\Soc (V))$ induced by the vertices $\Soc (X)$ where $X$ is ranging in the set $$\{X\leq V:X\not \subseteq \Soc (V){\ }{\ }and{\ }{\ }\Soc (V)\not \subseteq X\},$$ then
 $$\chi (\cG (V))\leq \chi (\cG (\Soc (V)))+(n-2)\chi (\widetilde{\cG })+m-1,$$ where
$m$ is the number of $R$-submodules of $V/\Soc (V).$
\end{pro}
\begin{proof} Suppose that $\cG (V)$ can be colored by finitely many colors. Part (c) is obvious, because $\cG (\Soc (V))$ is an induced subgraph of $\cG (V)$ meaning that, for any two vertices of $\cG (\Soc (V)),$ they are adjacent in $\cG (\Soc (V))$ if and only if they are adjacent in $\cG (V).$

It is clear that $V$ cannot have an infinite ascending or descending chain of submodules, because, otherwise, the proper submodules appearing in such a chain will be mutually adjacent vertices of the graph $\cG (V).$ Hence, (a) follows.

Part (b) is trivial. Otherwise, for any natural number $n,$ there is a simple module $S$ such that $V/S$ has at least $n+1$ submodules, implying that there are at least $n$ proper submodules of $V$ containing $S.$ These $n$ submodules are mutually adjacent vertices of $\cG (V)$ so that $\chi (\cG (V))\geq n-1$ for any $n.$

For the rest we assume that (a)-(c) hold. In particular, $V$ has a (finite) composition series, and so $\Soc (U)=\Soc (V)\cap U\neq 0$ for any nonzero submodule $U$ of $V.$

Let $\cP (V)$ be the set of proper submodules of $V.$ We define three sets:
$$\cA =\{X\in \cP (V):\Soc (V)\subseteq X\},{\ }{\ }{\ }{\ }\cB =\{X\in \cP (V):X\subseteq \Soc (V),X\neq \Soc (V)\}$$
$$\cC =\cP (V)-(\cA \cup\cB ),$$ so that $\cP (V)$ is the disjoint union $\cA \uplus \cB \uplus \cC.$

We now want to show that $\chi (\cG (V))$ has the desired upper bound. Firstly, as $X\cap Y\neq 0$ for any elements $X$ and $Y$ of $\cA ,$ we need to use at least $|\cA |$ colors in any coloring of $\cG (V).$ Moreover, as any element of $\cB $ is adjacent to every element of $\cA ,$ to color the vertices in $\cB $ we cannot use the colors that are already used to color the vertices in $\cA .$ Thus, we need at least $|\cA |+\chi (\cG (\Soc (V)))$ colors to color the vertices in $\cA \cup \cB ,$ and we can color the vertices in $\cA \cup \cB $ by using $|\cA |+\chi (\cG (\Soc (V)))$ colors, where, of course, $|\cA |=m-1.$

We will show that the vertices in $\cC $ can be colored by using $(n-2)\chi (\widetilde{\cG })$ colors which proves the desired upper bound for $\chi (\cG (V)).$

For any elements $Y_1$ and $Y_2$ of $\cC ,$ it is clear that $\Soc (Y_i)\in \cB ,$ and we see that $Y_1\cap Y_2\neq 0$ if and only if $\Soc (Y_1)\cap \Soc (Y_2)\neq 0$ (because, $Y_1\cap Y_2\neq 0$ implies $(Y_1\cap Y_2)\cap \Soc (V)\neq 0$). That is to say, $Y_1$ and $Y_2$ are adjacent vertices of $\cG (V)$ if and only if either $\Soc (Y_1)$ and $\Soc (Y_2)$ are equal or they are adjacent vertices of $\cG (\Soc (V)).$

Define a relation $\sim $ on $\cC $ as follows: $Y_1\sim Y_2$  if and only if $\Soc (Y_1)=\Soc (Y_2).$ Then, it is obvious that $\sim $ is an equivalence relation whose equivalence classes are precisely the sets
$$\cC _U=\{Z\in \cC :\Soc (Z)=U\}$$
where $U$ is ranging in the vertex set of $\widetilde{\cG }.$ Moreover, $\cC $ is the disjoint union of equivalence classes. Since any element in an equivalence class $\cC _U$ contains $U,$ part (b) implies that $|\cC _U|\leq n-2.$ For the same reason, the elements in a single equivalence class $\cC _U$ are mutually adjacent vertices in $\cG (V),$ an so we can color $\cC _U$ by using $n-2$ colors. Furthermore, it follows from the previous paragraph that, for any two distinct equivalence classes $\cC _U$ and $\cC _{U'},$ there is an element $Z\in \cC _U$ and an element $Z'\in \cC _{U'}$ such that $Z\cap Z'\neq 0$ if and only if $U$ and $U'$ are adjacent vertices of $\widetilde{\cG }.$ Consequently, we may color the vertices in $\cC $ by using $(n-2)\chi (\widetilde{\cG })$ colors.

Finally, we want to observe that if (a)-(c) hold then $\chi (\cG (V))$ is finite. Indeed, if (a)-(c) hold, then we have just proved that $\chi (\cG (V))$ has the mentioned upper bound in this result. As $\chi (\widetilde{\cG })\leq \chi (\cG (\Soc (V))),$ (b) and (c) imply that this upper bound is a finite number.
\end{proof}

For lower bounds we state.

\begin{rem}
\label{7.2} Assume that $\cG (V)$ can be colored by finitely many colors. Then, for any proper $R$-submodule $W$ of $V,$ $\cG (W)$ can be colored by finitely many colors, and
$$\chi (\cG (W))+m-1\leq \chi (\cG (V)),$$ where $m$ is the number of $R$-submodules of $V/W.$
\end{rem}
\begin{proof} Let $\cP (V)$ be the set of proper submodules of $V.$ Take any $W\in \cP (V).$ It is clear that $\cG (W)$ can be colored by finitely many colors, because $\cG (W)$ is an induced subgraph of $\cG (V).$ Moreover, part (a) of \ref{7.1} implies that there is a simple submodule of $W,$ and so part (b) of \ref{7.1} implies that the number of submodules of $V/W$ is finite. Define the sets
$$\cA =\{X\in \cP (V):W\subseteq X\},{\ }{\ }{\ }{\ }\cB =\{X\in \cP (V):X\subseteq W,X\neq W\}.$$
Imitating the relevant part of the proof of \ref{7.1} we easily see that we need at least $|\cA |+\chi (\cG (W))$ colors to color the vertices of $\cG (V)$ in $\cA \uplus \cB,$ finishing the proof.
\end{proof}

By using \ref{5.3} and \ref{7.1} we can find examples of modules $V$ such that $\cG (V)$ can be colored by finitely many colors and it has infinitely many vertices and edges. For instance, an $R$-module $V$ satisfying $\Soc (V)\cong S\oplus S$ and $V/\Soc (V)\cong T$ where $S$ and $T$ are nonisomorphic simple $R$-modules and $|\End _R(S)|$ is infinite.

We now specialize in semisimple modules.

\begin{cor}
\label{7.3} Let $V$ be a semisimple $R$-module isomorphic to $$n_1T_1\oplus n_2T_2\oplus \cdot \cdot \cdot \oplus n_rT_r$$ where $r$ is a natural number, $T_1,T_2,...,T_r$ are mutually nonisomorphic simple $R$-modules, and $n_1,n_2,...,n_r$ are natural numbers.
\begin{description}
 \item[{\rm (1)}] Let $r=1,$ so that $V\cong n_1T_1.$ Then, $\cG (V)$ can be colored by finitely many colors if and only if  either $n_1\leq 2,$ or $n_1\geq 3$ and $\End _R(T_1)$ is finite.
 \item[{\rm (2)}] Let $r\geq 2.$ Then, the following conditions are equivalent:
   \begin{description}
 \item[{\rm (i)}] $\cG (V)$ can be colored by finitely many colors.
 \item[{\rm (ii)}] $\cG (V)$ has finitely many vertices.
 \item[{\rm (iii)}] For each $i\in \{1,2,...,r\},$ if $n_i\geq 2$ then $|\End _R(T_i)|$ is finite.
\end{description}
\end{description}
\end{cor}
\begin{proof} (1) If $n_1\leq 2,$ then any nonzero submodule of $V$ is simple so that $\cG (V)$ has no edges. So, in this case $\cG (V)$ can be colored by a single color.

Assume now that $n_1\geq 3.$ Suppose that $\cG (V)$ can be colored by finitely many colors. Part (b) of \ref{7.1} implies that $(n_1-1)T_1$ has finitely many submodules. Then, as $n_1-1\geq 2,$ it follows from \ref{5.3} that $|\End _R(T_1)|$ is finite.

Conversely, if $n_1\geq 3$ and $|\End _R(T_1)|$ is finite, then \ref{5.3} implies that $\cG (V)$ has finitely many vertices.

(2) (ii) $\iff $ (iii): Follows from \ref{5.3}.

(ii) $\implies $ (i): Obvious.

(i) $\implies $ (ii): Follows from part (b) of \ref{7.1} and \ref{5.3}.
\end{proof}

Suppose that $V$ is a semisimple $R$-module having a (finite) composition series and that $V$ is not isomorphic to $S\oplus S$ for any simple $R$-module $S.$ We see in the above result that $\cG (V)$ can be colored by finitely many vertices if and only if $\cG (V)$ has finitely many vertices.

The results \ref{7.1} and \ref{7.3} complete the description of the structure of $V$ for which $\cG (V)$ can be colored by finitely many colors.

\begin{lem}
\label{7.4} Let $V$ be an $R$-module having a composition series of length $n$ where $n$ is a natural number. For any nonnegative integer $i\leq n,$ let $\cS _i(V)$ denote the set of all $R$-submodules of $V$ of composition length $i.$ Let $k$ be a natural number. Then:
\begin{description}
 \item[{\rm (1)}] Let $n=2k-1$ be odd. Then, $X\cap Y\neq 0$ for any elements $X$ and $Y$ of the set $\ds \biguplus _{i=k}^n\cS _i(V).$ In particular, to color $\cG (V)$ we need to use at least $\ds \sum _{i=k}^{n-1}|\cS _i(V)|$ colors.
 \item[{\rm (2)}] Let $n=2k$ be even. Then, $X\cap Y\neq 0$ for any elements $X$ and $Y$ of the set $\ds \biguplus _{i=k+1}^n\cS _i(V).$ Moreover, given any $U\in \cS _k(V),$ then $U\cap W\neq 0$ for any $W$ in $\ds \biguplus _{i=k+1}^n\cS _i(V).$ In particular, $\ds \sum _{i=k+1}^{n-1}|\cS _i(V)|$ colors are not enough to color $\cG (V).$
\end{description}
\end{lem}
\begin{proof} Let $\ell ({\ })$ denote the composition length of its argument. This is obvious, because, for any two submodules $A$ and $B$ of $V,$ if $\ell (A)+\ell (B)>\ell (V)$ then $A\cap B\neq 0.$ (If it is not obvious, see, for instance, \cite[Chapter 4]{Warf}).
\end{proof}

We finally proceed to obtain the chromatic number of a semisimple module. We first need a simple observation.

\begin{lem}
\label{7.5} Let $V$ be a semisimple $R$-module having a composition series of length $2k$ where $k$ is a natural number. Suppose that there is a simple $R$-module whose multiplicity as a composition factor of $V$ is odd. Then, given any three $R$-submodules $W_1,W_2,W_3$ of $V$ of composition length $k,$  at least one of the three intersection
$$W_1\cap W_2,{\ }{\ }{\ }W_1\cap W_3,{\ }{\ }{\ }W_2\cap W_3$$ is nonzero.
\end{lem}
\begin{proof} Otherwise, there are submodules $W_1,W_2,W_3$ of $V$ of composition length $k$ such that $W_i\cap W_j=0,$ and hence $W_i\oplus W_j=V,$ for any $i,j$ with $i\neq j.$ Now, we have the following isomorphisms of modules:
$$W_1\cong (W_1\oplus W_2)/W_2=(W_2\oplus W_3)/W_2\cong W_3\cong (W_1\oplus W_3)/W_1=V/W_1.$$ But then, $W_1\cong V/W_1$ implies that the multiplicity of any simple module as a composition factor of $V$ is even.
\end{proof}

\begin{pro}
\label{7.6} Let $V$ be a semisimple $R$-module isomorphic to $$n_1T_1\oplus n_2T_2\oplus \cdot \cdot \cdot \oplus n_rT_r$$ where $r$ is a natural number, $T_1,T_2,...,T_r$ are mutually nonisomorphic simple $R$-modules, and $n_1,n_2,...,n_r$ are natural numbers. Assume that, for any natural number $i$ with $1\leq i\leq r,$ if $n_i\geq 2$ then $d_i=|\End _R(T_i)|$ is finite. Let $\cS _j(V)$ denote the set of all $R$-submodules of $V$ of composition length $j$ where $j$ is a nonnegative integer. Let $n=n_1+n_2+\cdot \cdot \cdot +n_r,$ and let $k$ be a natural number. Suppose that $n>2.$ Then:
\begin{description}
 \item[{\rm (1)}] $\cG (V)$ can be colored by finitely many colors.
 \item[{\rm (2)}] If $n=2k-1$ is odd, then
 $$\chi (\cG (V))= \sum _{i=k}^{n-1}|\cS _i(V)|.$$
 \item[{\rm (3)}] If $n=2k$ is even and at least one of $n_1,n_2,...,n_r$ is odd, then
 $$\chi (\cG (V))= \frac{1}{2}|\cS _k(V)|+\sum _{i=k+1}^{n-1}|\cS _i(V)|.$$
\end{description}
\end{pro}
\begin{proof} (1) Follows from \ref{7.3}.

(2) Let us define $\cE $ and $\cF $ as
$$\cE =\biguplus_{i=k}^{n-1}\cS _i(V),{\ }{\ }{\ }\cF =\biguplus_{i=1}^{k-1}\cS _i(V).$$ Note that the set of vertices in $\cG (V)$ is the disjoint union $\cE \uplus \cF .$ It follows from part (1) of \ref{7.4} that the vertices in $\cE $ are mutually adjacent. So, in any coloring of $\cG (V),$ the vertices in $\cE $ must be colored by using mutually different colors. We want to show that the vertices in $\cF $ can be colored by using the colors that are used to color the vertices in $\cE .$ This will imply that the chromatic number of $\cG $ is $|\cE |,$ which finishes the proof.

From part (1) of \ref{6.7}, for any nonnegative integer $j\leq n,$ there is a bijection $$\phi _j:\cS _j(V)\rightarrow \cS _{n-j}(V)$$ such that $X\cap \phi _j(X)=0$ for any $X\in \cS _j(V).$ Then, these maps induce the bijection
$$\varphi =\biguplus_{i=k}^{n-1}\phi _i :\cE \rightarrow \cF $$ such that $Y\cap \varphi (Y)=0$ for any $Y$ in $\cE .$ Hence, for any vertex $Z$ in $\cF ,$ we may color $Z$ by using the same color of the vertex $\varphi ^{-1}(Z)$ in $\cE ,$ where $\varphi ^{-1}$ denotes the inverse of $\varphi .$

Suppose we have colored the vertices as we explained. We want to observe that no two adjacent vertices are colored by the same color. Take any two vertices $Z_1$ and $Z_2$ of $\cG (V)$ whose colors are the same. It is clear that exactly one of $Z_1$ and $Z_2$ must be in $\cE ,$ say $Z_1\in \cE $ and $Z_2\in \cF .$ Then, $\varphi ^{-1}(Z_2)$ and $Z_1$ are both in $\cE $ and their colors are the same. As the vertices in $\cE $ are colored by using mutually different colors, $\varphi ^{-1}(Z_2)=Z_1,$ implying that $Z_1\cap Z_2=0.$

(3) Let us define $\cE '$ and $\cF '$ as follows:
$$\cE '=\biguplus_{i=k+1}^{n-1}\cS _i(V),{\ }{\ }{\ }\cF '=\biguplus_{i=1}^{k-1}\cS _i(V).$$ Note that the set of vertices in $\cG (V)$ is the disjoint union $\cE '\uplus \cF ' \uplus \cS _k(V).$ Arguing as in the proof of the previous part, we see that the vertices in $\cE '\uplus \cF '$ can be colored by $|\cE '|$ colors, and $|\cE '|$ is the smallest possible number of colors to be used to color the vertices in $\cE '\uplus \cF ',$ and the vertices in $\cE '$ are colored by mutually different colors. It follows from part (2) of \ref{7.4} that any arbitrarily chosen vertex in $\cS _k(V)$ is adjacent to every vertex in $\cE '.$ To color the vertices in $\cS _k(V)$ we must use colors that are not already used to color the vertices in $\cE '\uplus \cF '.$ Hence, to finish the proof we should show that the smallest number of colors to color the vertices in $\cS _k(V)$ is $\frac{1}{2}|\cS _k(V)|.$

From part (2) of \ref{6.7}, there are disjoint sets $\cA $ and $\cB $ and a bijection $\alpha :\cA \rightarrow \cB $ such that $\cS _k(V)=\cA \uplus \cB $ and $X\cap \alpha (X)=0$ for any $X\in \cA .$ Let $|\cA |=m$ for some natural number $m,$ so $\cS _k(V)=2m.$

We now want to show that to color the vertices in $\cS _k(V)$ we need at least $m$ colors. Indeed, suppose for a moment that, it is possible to color all the $2m$ vertices in $\cS _k(V)$ by using less than $m$ colors. Then there must be at least three vertices with the same color. However, this is not possible by the virtue of \ref{7.5}.

We finally want to show that we can color the vertices in $\cS _k(V)$ by using $m$ colors. Indeed, we may write $\cS _k(V)$ as a disjoint union
$$\cS _k(V)=\biguplus _{X\in \cA}\{X,\alpha (X)\}.$$ Let $\{\cU _X:X\in \cA\}$ be a set of $m$ mutually distinct colors indexed by the elements of $\cA .$ It is now clear that we can color the vertices in $\cS _k(V)$ by $m$ colors, by insisting that, for each $X\in \cA ,$ the colors of the two vertices in $\{X, \alpha (X)\}$ are the same which is $\cU _X.$
\end{proof}
The previous result does not cover the case in which all of the numbers $n,n_1,n_2,...,n_r$ are even. We do not know what the chromatic number is in this case. The cardinalities $\cS _i(V)$ appeared in the previous result can be calculated by using \ref{6.4}.

\section{Planarity}
A graph is called planar if it can be drawn on the plane in such a way that its edges intersect only at their endpoints. See, for instance, \cite{West} or \cite{Ore}. The planarity of intersection graphs of ideals of a commutative ring is studied in \cite{Jaf3}. We devote this section to the determination of the structure of the modules whose intersection graphs are planar. We have no restrictions on the module $V$ and the ring $R$ so that some of our results generalize the known results in the special cases.

For any natural number $n$ we denote by $\K_ n$ the complete graph on $n$ vertices, which is not planar for $n\geq 5,$ see \cite{West}. If a graph $\cG $ has a subgraph isomorphic to $\K _n,$ then we say that $\cG $ contains $\K _5.$ Thus, $\cG $ does not contain $\K _n$ if and only if there are no mutually adjacent $n$ vertices in $\cG .$

\begin{rem}
\label{8.1} If $\cG (V)$ does not contain $\K _n,$ then $V$ has a composition series of length less than or equal to $n.$
\end{rem}
\begin{proof} This is obvious, because the subgraph of $\cG (V)$ induced by the proper submodules of $V$ appearing in any ascending or descending chain of submodules of $V$ is a complete graph.
\end{proof}

\begin{lem}
\label{8.2} Let $V$ be an $R$-module having a composition series of length $m$ where $m\geq 3$ is a natural number. Suppose that $\cG (V)$ does not contain $\K _n.$ Then:
\begin{description}
 \item[{\rm (1)}] If $V$ is not a semisimple $R$-module, then the number of maximal $R$-submodules of $V$ is at most $n-m+1.$
 \item[{\rm (2)}] If $V$ is a semisimple $R$-module, then the number of maximal $R$-submodules of $V$ is at most $n-m+2.$
 \item[{\rm (3)}] If $2m>n+2,$ then $V$ is not a semisimple $R$-module.
\end{description}
\end{lem}
\begin{proof} (1) As $V$ is not semisimple and has a (finite) composition series, the Jacobson radical $\Jac (V)$ of $V$ is not zero. Suppose for a moment that there are at least $n-m+2$ maximal submodules. Let $\cM $ be the set of all maximal submodules of $V.$ Let $\cN $ be the set of all proper submodules of $V$ appearing in a composition series of $V$ which contains $\Jac (V)$ as one of its term. Therefore, $\cN $ is a chain with respect to the subset or equal to relation, $|\cN |=m-1,$ $\Jac (V)\in \cN ,$ and $|\cM \cap \cN |=1.$ Since $\Jac (V)$ is the intersection of all the maximal submodules of $V,$ the intersection of any two elements of the set $\cM \cup \cN $ is nonzero. This implies that $\cG (V)$ contains $\K _n,$ because $|\cM \cap \cN |=1$ and so $$|\cM \cup \cN |\geq (n-m+2)+(m-1)-1=n.$$

(2) We argue as in the previous part. Let $\cM $ be the set of all maximal submodules of $V,$ and $\cN $ be the set of all proper submodules of $V$ appearing in a composition series $\cC$ of $V.$ Suppose that the conclusion is false. Then, $V$ has at least $n-m+3$ maximal submodules. We cannot deduce that the intersection of any two elements of the set $\cM \cup \cN $ is nonzero, because $\Jac (V)=0$ here. However, as $m>2,$ given any maximal submodule of $V$ and any nonsimple submodule of $V,$ the sum of their composition lengths is greater than $m,$ so that their intersection is nonzero. Therefore, we deduce that the intersection of any two elements of the set $\cM \cup (\cN -\{W\})$ is nonzero, where $W$ is the (unique) simple module appearing in the chain $\cC .$ As $|\cM \cap (\cN -\{W\})|=1,$ we see that $|\cM \cup (\cN -\{W\})|\geq n,$ implying that $\cG (V)$ contains $\K _n.$

(3) Suppose that $V$ is semisimple. As $V$ has a composition series of length $m,$ the semisimplicity of $V$ implies that $V$ has at least $m$ maximal submodules. (If it is not obvious, see \ref{5.2}, or write $V$ as a direct sum of $m$ simple submodules, say $S_1,S_2,...,S_m,$ and then letting, $M_i$ be the sum of all simple modules $S_1,S_2,...,S_m$ except $S_i,$ we see that $M_1,M_2,...,M_m$ are mutually distinct $m$ maximal submodules). By part (2) we should have $m\leq n-m+2.$
\end{proof}

In the rest of this section, we will frequently mention about the length of a composition series and the number of maximal submodules of a module. To simplify the reading we develop some notations. Let $W$ be an $R$-module having a (finite) composition series. We sometimes use the notations $\ell (W)$ and $\fm(W)$ to denote the composition length of $W$ and the number of maximal $R$-submodules of $W,$ respectively. Moreover, for any nonnegative integer $n$ and any $R$-module $U,$ by writing $\ell (U)=n,$ we claim that $U$ has a composition series of length $n.$

We may restate \ref{3.2} as follows.

\begin{pro}
\label{8.3} $\cG (V)$ does not contain $\K _3$ if and only if $\ell (V)\leq 2,$ or $\ell (V)=3$ and $\fm (V)=1.$
\end{pro}

\begin{pro}
\label{8.4} $\cG (V)$ does not contain $\K _4$ if and only if one of the following conditions holds:
\begin{description}
 \item[{\rm (a)}] $\ell (V)\leq 2.$
 \item[{\rm (b)}] $\ell (V)=3,$ and $V$ is not a semisimple $R$-module, and $\fm (V)\leq 2.$
 \item[{\rm (c)}] $V$ is a semisimple $R$-module isomorphic to a direct sum of three mutually nonisomorphic simple $R$-modules. In particular, $\ell (V)=3.$
 \item[{\rm (d)}] $\ell (V)=4,$ and $V$ has a unique maximal $R$-submodule $J,$ and $\fm (J)=1.$
\end{description}
\end{pro}
\begin{proof} Suppose that $\cG (V)$ does not contain $\K _4.$ We know from \ref{8.1} that $\ell (V)\leq 4.$ We now want to understand what happens at each possible value for $\ell (V).$

Assume that $\ell (V)=3:$

(b) follows from \ref{8.2}. Assume now that $V$ is semisimple, so that $V\cong S_1\oplus S_2\oplus S_3$ for some simple $R$-modules $S_i.$ Then \ref{8.2} implies that $\fm (V)\leq 3.$ If any two of $S_1,S_2,S_3$ are isomorphic, it follows from \ref{5.2} that $\fm (V)>3.$ Thus, (c) follows.

Assume that $\ell (V)=4:$

Any semisimple module of composition length $4$ has at least $4$ maximal submodules. (See \ref{5.2}, or the proof of part (3) of \ref{8.2}). Thus, \ref{8.2} implies that $V$ is not semisimple and $V$ has a unique maximal submodule, say $J.$ This shows that the proper submodules of $V$ are precisely $J$ and the proper submodules of $J.$ Hence, $\cG (J)$ must not contain $\K _3$ (because, otherwise the vertices of $\K _3$ in $\cG (J)$ and $J$ form $\K _4$ in $\cG (V)$). As $\ell (J)=3,$ it follows from \ref{8.3} that $\fm (J)=1.$ Thus, (d) follows.

Conversely, we want to observe that for any module $V,$ satisfying any of the conditions (a)-(d), the graph $\cG (V)$ does not contain $K_4.$

It is clear that if $V$ satisfies (a) then $\cG (V)$ has no edges.

Moreover, if $V$ satisfies (b) or (d), then it is obvious that $V$ has at most $2$ nonsimple proper submodules. As the intersection of any two distinct simple modules is zero, this shows that $\cG (V)$ does not contain $\K _4.$

Finally, suppose that $V$ satisfies (c). By using \ref{4.1} or \ref{6.4}, we see that $V$ has $3$ maximal submodules and $3$ simple submodules. So $\cG (V)$ does not contain $\K _4,$ because there are no edges between simple submodules and the intersection of all maximal submodules is zero.
\end{proof}

\begin{pro}
\label{8.5} $\cG (V)$ does not contain $\K _5$ if and only if one of the following conditions holds:
\begin{description}
 \item[{\rm (a)}] $\ell (V)\leq 2.$
 \item[{\rm (b)}] $\ell (V)=3,$ and $V$ is not a semisimple $R$-module, and $\fm (V)\leq 3.$
 \item[{\rm (c)}] $V$ is a semisimple $R$-module isomorphic to a direct sum of three mutually nonisomorphic simple $R$-modules. In particular, $\ell (V)=3.$
 \item[{\rm (d)}] $V$ is a semisimple $R$-module isomorphic to $S\oplus S\oplus T$ where $S$ and $T$ are nonisomorphic simple $R$-modules with $|\End _R(S)|=2.$ In particular, $\ell (V)=3.$
 \item[{\rm (e)}] $\ell (V)=4,$ and $V$ has a unique maximal $R$-submodule $J,$ and $\fm (J)\leq 2.$
 \item[{\rm (f)}] $\ell (V)=4,$ and $V$ has a unique maximal $R$-submodule $J,$ and $J$ is a semisimple $R$-module isomorphic to a direct sum of three mutually nonisomorphic simple $R$-modules.
 \item[{\rm (g)}] $\ell (V)=4,$ and $V$ has exactly two maximal $R$-submodule $J$ and $J',$ and $\fm (J)=1,$ and $\fm (J')=1.$
 \item[{\rm (h)}] $\ell (V)=5,$ and $V$ has a unique maximal $R$-submodule $J_1,$ and $J_1$ has a unique maximal $R$-submodule $J_2,$ and $J_2$ has a unique maximal $R$-submodule.
\end{description}
\end{pro}
\begin{proof} We argue as in the proof \ref{8.4}. Suppose that $\cG (V)$ does not contain $\K _5.$ Then, $\ell (V)\leq 5$ from \ref{8.1}. We now want to understand what happens at each possible value for $\ell (V).$

Assume that $\ell (V)=3:$

(b) is an easy consequence of \ref{8.2}.

Assume that $V$ is semisimple, so that $V\cong S_1\oplus S_2\oplus S_3$ for some simple $R$-modules $S_i.$ Now \ref{8.2} implies that $\fm (V)\leq 4.$ Let $d=|\End _R(S_1)|.$ By using \ref{5.2} we calculate $\fm (V)$ as:

$$\fm (V)=
\left\{
  \begin{array}{ll}
   3 , & \hbox{if $S_1,S_2,S_3$ are mutually nonisomorphic;} \\
   d+2 , & \hbox{if $S_1\cong S_2$ but $S_1\not \cong S_3$;} \\
   1+d+d^2 , & \hbox{if $S_1\cong S_2\cong S_3.$}
  \end{array}
\right.$$
Thus, (c) and (d) follow because the smallest possible value for $d$ is $2$ (being a division ring, if $\End _R(S_1)$ is finite then it must be a finite field).

Assume that $\ell (V)=4:$

From \ref{8.2} we conclude that $V$ is not semisimple and $\fm (V)\leq 2.$ So, $\fm (V)=1$ or $\fm (V)= 2.$ Suppose first that $\fm (V)=1.$ Then, $V$ has a unique maximal submodule $J,$ which must contain every proper submodule of $V.$ As $\cG (V)$ does not contain $\K _5,$ we see that $\cG (J)$ cannot contain $\K _4.$ Hence, (e) and (f) follows from \ref{8.4}. Note that the module $J$ in (e) is necessarily nonsemisimple.

Suppose now that $\fm (V)=2.$ Then, $V$ has exactly two maximal submodules, say $J$ and $J'$. We will show that $\fm (J)=1=\fm (J').$ Suppose for a moment that $J$ has two maximal submodules $X$ and $Y.$ Considering the composition lengths, if $J\cap J'\notin \{X,Y\}$ then the vertices in the set $$\{J,J',X,Y,J\cap J'\}$$ form $\K _5$ in $\cG (V).$ On the other hand, if $J\cap J'\in \{X,Y\}$ then the vertices in the set $$\{J,J',X,Y,X\cap Y\}$$ form $\K _5$ in $\cG (V).$ Hence, $\fm (J)=1$ (and similarly $\fm (J')=1$). Consequently, (g) follows.

Assume that $\ell (V)=5:$

It follows from \ref{8.2} that $V$ has a unique maximal submodule $J_1.$ As every proper submodule of $V$ is contained in $J_1,$ the graph $\cG (J_1)$ cannot contain $\K _4.$ Then, as $\ell (V)=4,$ (h) follows from \ref{8.4}.

Conversely, we need to observe that $\cG (V)$ does not contain $\K _5$ for any $V$ satisfying any of the conditions (a)-(h). However, this is easy to see because the conditions (a)-(h) give enough information to draw the graphs $\cG (V)$ in each cases. More than this, we actually see that $\cG (V)$ are planar in each of the conditions (a)-(h)
\end{proof}

If $\cG (V)$ is planar, then $\cG (V)$ does not contain $\K _5,$ and so $V$ satisfies one of the conditions in the above result. On the other hand, in each of the conditions (a)-(h) above, it is easy to draw $\cG (V)$ and observe that $\cG (V)$ is planar. Thus we have the following result.

\begin{cor}
\label{8.6} $\cG (V)$ is planar if and only if $V$ satisfies one of the conditions (a)-(h) given in \ref{8.5}.
\end{cor}

\end{document}